\documentclass [a4paper,12pt]{article}
\usepackage{amsmath}
\usepackage{amsfonts}
\usepackage{indentfirst}
\usepackage{esint}
\usepackage{latexsym}
\usepackage{bbm}
\usepackage{amsfonts}
\usepackage{mathrsfs}
\usepackage{amssymb}
\usepackage{amsmath}
\usepackage{amsthm}
\usepackage{latexsym}
\usepackage{indentfirst}
\usepackage{enumitem}
\usepackage{bm}
\usepackage{esint}
\usepackage{cite}
\numberwithin{equation}{section}
\usepackage{color}
\allowdisplaybreaks
\usepackage{amsthm}
\usepackage{amssymb}
\usepackage{enumerate}
\usepackage{ulem}
\allowdisplaybreaks

\newtheorem{theorem}{Theorem}[section]
\newtheorem{lemma}[theorem]{Lemma}
\newtheorem{thm}[theorem]{Theorem}

\newtheorem{cor}[theorem]{Corollary}

\newtheorem{rmk}[theorem]{Remark}

\makeatletter
\usepackage{amsmath}
\usepackage{amsfonts}
\usepackage{indentfirst}
\usepackage{esint}
\usepackage{latexsym}
\usepackage{color}

\usepackage{amsthm}
\usepackage{amssymb}
\usepackage{enumerate}
\usepackage{ulem}
\makeatletter

\newcommand{\Rmnum}[1]{\expandafter\@slowromancap\romannumeral #1@}
\makeatother
\begin{document}
\title{Approximate Two-Sphere One-Cylinder Inequality in Parabolic Periodic Homogenization}
\author{Yiping Zhang\footnote{Email:zhangyiping161@mails.ucas.ac.cn}\\Academy of Mathematics and Systems Science, CAS;\\
University of Chinese Academy of Sciences;\\
Beijing 100190, P.R. China.}
\date{}
\maketitle
\begin{abstract}
In this paper, for a family of second-order parabolic equations with rapidly oscillating and time-dependent periodic coefficients, we are interested in an approximate two-sphere one-cylinder inequality for these solutions in parabolic periodic homogenization, which implies an approximate quantitative propagation of smallness. The proof relies on the asymptotic behavior of fundamental solutions and the Lagrange interpolation technique.
\end{abstract}

\section{Introduction}
Quantitative propagation of smallness is one of the  central issues in the quantitative study of solutions of elliptic and parabolic equations. It can be
stated as follows: a solution $u$ of a PDE $Lu=0$ on a domain $X$ can be made arbitrarily small on any given compact subset of $X$ by making it sufficiently small on an arbitrary given subdomain $Y$. There are many important applications in quantitative propagation of smallness, such as the stability estimates for the Cauchy problem \cite{alessandrini2009stability} and the Hausdorff measure estimates of nodal sets of eigenfunctions \cite{lin1991nodal}, \cite{logunov2018nodal}.

For solutions of second order parabolic equations
\begin{equation}
\left(\partial_t-\mathcal{L}\right)u=\partial_tu-\partial_i\left(a_{ij}(x,t)\partial_j u\right)+b_i\partial_i u+cu=0,
\end{equation}(the summation convention is used throughout the paper). There exists a large literature on the three cylinder
inequality for solutions to parabolic equations. If $A(x,t)=\left(a_{ij}(x,t)\right)$ is three times continuously differentiable with respect to $x$ and one time continuously differentiable with respect to $t$, and $b$ and $c$ are bounded, then a not optimal three-cylinder inequality has been obtained in \cite{glagoleva1967some} and \cite{varin1992three-cylinder}. In \cite{glagoleva1967some}, the three-cylinder inequality is derived by the Carleman estimates proved in \cite{Li1963An}. In 2003, Vessella\cite{vessella2003carleman} has obtained the following optimal three-cylinder inequality:

\begin{equation}\left\|u
\right\|_{L^{2}(Q_R^{T/2})} \leqslant C\left( ||u||_{L^2(Q_\rho^{T})}\right)^{\kappa_\rho}\left( ||u||_{L^2(Q_{R_0}^{T})}\right)^{\kappa_\rho},\end{equation}
for every $0<\rho<R<R_0$, under the assumptions that the derivatives $\partial A/\partial t$, $\partial A/\partial x^i$, $\partial^2 A/(\partial t\partial
x^j)$, $\partial^2 A/(\partial x^i\partial x^j)$, for every $i,j\in \{1,\cdots,d\}$, $b=(b_1,\cdots,b_d)$ and $c$ are bounded, where $Q_r^t=B_r\times (-t,t)$
with $B_r$ the $d$-dimensional open ball centered at $0$ of radius $r$ and $\kappa_\rho\sim |\log \rho|^{-1}$ as $\rho\rightarrow 0$. The optimality of Equation $(1.2)$ consists in the growth rate of the exponent $\kappa_\rho$. Later on, Escauriaza and Vessella \cite{escauriaza2003optimal} have obtained the inequality $(1.2)$ under the assumptions that $A\in C^{1,1}(\mathbb{R}^{n+1})$ , $b$ and $c$ are bounded. Moreover, Vessella \cite{vessella2008quantitative} has obtained the following two-sphere one-cylinder inequality:
\begin{equation}\left\|u\left(., t_{0}\right)\right\|_{L^{2}\left(B_{\rho}\left(x_{0}\right)\right)} \leqslant C\left\|u\left(., t_{0}\right)\right\|_{L^{2}\left(B_{r}\left(x_{0}\right)\right)}^{\theta}\|u\|_{L^{2}\left(B_{R}\left(x_{0}\right) \times\left(t_{0}-R^{2}, t_{0}\right)\right)}^{1-\theta},\end{equation}under the assumptions that $A$ satisfies the Lipschitz continuity: $|A(y,s)-A(x,t)|\leq C\left(|x-y|+|t-s|^{1/2}\right)$, $b=0$ and $c=0$, where $\theta=\left(C\log {\frac{R}{Cr}}\right)^{-1}$, $0<r<\rho<R$ and $C$ depends neither on $u$ nor on $r$ but may depend on $\rho$ and $R$, see also \cite{escauriaza2006doubling}, where the two-sphere one-cylinder inequality $(1.3)$ for time-dependent parabolic operators was first established. And the estimate $(1.3)$ has first been obtained by Landis and
Oleinik \cite{landis1974generalized}, when $A$ does not depend on $t$.

In general, the Carleman estimates are tools often used to obtain a three-cylinder inequality and the unique continuation properties for solutions. The Carleman estimates are weighted integral inequalities with suitable weight functions satisfying some convexity properties. The three-cylinder inequality is obtained by applying the Carleman estimates by choosing a suitable function. For Carleman estimates and the unique continuation properties for the parabolic operators, we refer readers to  \cite{escauriaza2000carleman,escauriaza2004unique,escauriaza2001carleman,fernandez2003unique,vessella2008quantitative,koch2009carleman} and their references therein for more results.

Recently, Guadie and Malinnikova \cite{guadie2014on}
developed the three-ball inequality with the help of Poisson kernel for harmonic functions. This method also has been used in \cite{kenig2019propagation} to obtain an approximate three-ball inequality in elliptic periodic homogenization. Moreover, it is an interesting problem to extend the Carleman estimates to the homogenization equations and left for the future.

In this paper, we intend to develop an approximate two-sphere one-cylinder inequality for the $L^\infty$-norm in parabolic periodic homogenization equation, parallel to the inequality $(1.3)$, with the different exponent $\theta$.

We consider a family of second-order parabolic equations in divergence form with rapidly oscillating and time-dependent periodic coefficients,
\begin{equation}
\partial_t u_\varepsilon-\operatorname{div}\left(A(x/\varepsilon,t/\varepsilon^2)\nabla u_\varepsilon\right)=0,
\end{equation}
where $1>\varepsilon>0$ and $A(y,s)=(a_{ij}(y,s))$ is a symmetric $d\times d$ matrix-valued function in $\mathbb{R}^d\times \mathbb{R}$ for $d\geq 2$. Assume that $A(y,s)$ satisfies the following assumptions:

(i) Ellipticity: For some $0<\mu<1$ and all $(y,s)\in \mathbb{R}^d\times \mathbb{R}$, $\xi\in \mathbb{R}^d$, it holds that
\begin{equation}
\mu |\xi|^2\leq A(y,s)\xi\cdot\xi\leq \mu^{-1}|\xi|^2.
\end{equation}

(ii) 1-Periodicity:
\begin{equation}A(y+z,s+t)=A(y,s) \quad \text{for } (y,s)\in \mathbb{R}^d\times \mathbb{R} \text{ and }(z,t)\in\mathbb{Z}^d\times\mathbb{Z}.\end{equation}

(iii)H\"{o}lder continuity: There exist constants $\tau>0$ and $0<\lambda<1$ such that
\begin{equation}|A(x,t)-A(y,s)|\leq \tau \left(|x-y|+|t-s|^{1/2}\right)^\lambda\end{equation} for any $(x,t),(y,s)\in \mathbb{R}^d\times \mathbb{R}$.

We are able to establish the following approximate two-sphere one-cylinder inequality in ellipsoids. The definition of ellipsoids $E_r$ depending on the coefficients $A(y,s)$ is given in Section 2.
\begin{thm}(Interior two-sphere one-cylinder inequality) Let $u_\varepsilon$ be a solution of $(1.4)$ in $B_{{R}}\times (-T,T)$. For $0<r_1<r_2<{r_3}/12<R/8$, then there holds
\begin{equation}
\sup_{E_{r_2}}|u_\varepsilon(\cdot,t_0)|\leq C\left\{ (\sup_{E_{r_1}}|u_\varepsilon(\cdot,t_0)|)^{\alpha}(\sup_{\tilde{\Omega}_{{r_3},t_0}}|u_\varepsilon|)^{1-\alpha}+\frac{r_3}{r_{1}}\left[\frac{\varepsilon}{{r_3}}\log (2+\frac{r_3}{{\varepsilon}})\right]^{\alpha}\sup_{\tilde{\Omega}_{{r_3},t_0}}|u_\varepsilon|\right\},
\end{equation} where $\alpha=\frac{\log \frac{{r_3}}{4r_2}}{\log \frac{{r_3}}{2r_1}}$, and $C$ depends only on $d$, $\mu$ and $(\tau,\lambda)$, and $\tilde{\Omega}_{{r_3},t_0}=E_{r_3} \times (t_0-{r_3}^2,t_0)$ is a subdomain of $B_{R}\times (-T,T)$ with $R$ and $T$ fixed.
\end{thm}

A direct consequence of Theorem 1.1 is the following approximate two-sphere one-cylinder inequality in balls.
\begin{cor}
Let $u_\varepsilon$ be a solution of $(1.4)$ in $B_{R}\times (-T,T)$. For $0<r_1<r_2<\mu {r_3}/12<\mu R/8$, then there holds
\begin{equation}
\sup_{B_{r_2}}|u_\varepsilon(\cdot,t_0)|\leq C\left\{ (\sup_{B_{r_1}}|u_\varepsilon(\cdot,t_0)|)^{\alpha}(\sup_{{\Omega}_{{r_3},t_0}}|u_\varepsilon|)^{1-\alpha}+\frac{r_3}{r_{1}}\left[\frac{\varepsilon}{{r_3}}\log (2+\frac{r_3}{{\varepsilon}})\right]^{\alpha}\sup_{{\Omega}_{{r_3},t_0}}|u_\varepsilon|\right\},
\end{equation} where $\alpha=\frac{\log \frac{C_1 {r_3}}{r_2}}{\log \frac{{r_3}}{2r_1}}$, and $C$ depends only on $d$, $\mu$ and $(\tau,\lambda)$ and and ${\Omega}_{{r_3},t_0}=B_{r_3} \times (t_0-r_3^2,t_0)$ is a subdomain of $B_{R}\times (-T,T)$ with $R$ and $T$ fixed.
\end{cor}

\begin{rmk}
Compared with the Lipschitz regularity needed to obtain the inequality $(1.3)$, only H\"{o}lder continuity is imposed to obtain the inequality $(1.9)$, with the different exponent $\theta$. Moreover, as $\varepsilon\rightarrow 0$ with $r_3$ fixed, the inequality $(1.9)$ converges to the standard two-sphere one-cylinder inequality $(1.3)$ (with the different exponent $\theta$). However, if $r_3\backsim \varepsilon$, then the inequality $(1.9)$ gives us nothing, since in the $(\varepsilon,\varepsilon^2)$-scale, the operator $\partial_t -\operatorname{div}\left(A(x/\varepsilon,t/\varepsilon^2)\nabla \cdot\right)$ behaves like the classical operator $\partial_t -\operatorname{div}\left(A(x,t)\nabla \cdot\right)$ after a change of variables, where the Lipschitz regularity needed to obtain the two-sphere one-cylinder inequality.
\end{rmk}
This interpolation method may also apply to the parabolic equation with potential. Namely, let $u_\varepsilon$ satisfies the following equation
\begin{equation}
\partial_t u_\varepsilon-\operatorname{div}\left(A(x/\varepsilon,t/\varepsilon^2)\nabla u_\varepsilon\right)+V^\varepsilon u_\varepsilon=0,
\end{equation}
where $V^\varepsilon=V(x/\varepsilon)$
with $V$ being 1-periodic and $V\in L^{\frac{d+2}{2}}(Z)$ with $Z=[0,1)^d$. Note that $V$ is independent of the variable $t$. Consequently,  we are able to establish the following approximate two-sphere one-cylinder
 inequality in ellipsoids for the solution to $(1.10)$.
\begin{thm}(Interior two-sphere one-cylinder inequality) Let $u_\varepsilon$ be a solution of $(1.10)$ in $B_{{R}}\times (-T,T)$. For $0<r_1<r_2<{r_3}/12<R/8$, and $\varepsilon\leq r_3$, then there holds
\begin{equation}
\sup_{E_{r_2}}|u_\varepsilon(\cdot,t_0)|\leq C\left\{ (\sup_{E_{r_1}}|u_\varepsilon(\cdot,t_0)|)^{\alpha}(\sup_{\tilde{\Omega}_{{r_3},t_0}}|u_\varepsilon|)^{1-\alpha}+\frac{r_3}{r_{1}}\left[\frac{\varepsilon}{{r_3}}\log (2+\frac{r_3}{{\varepsilon}})\right]^{\alpha}\sup_{\tilde{\Omega}_{{r_3},t_0}}|u_\varepsilon|\right\},
\end{equation} where $\alpha=\frac{\log \frac{{r_3}}{4r_2}}{\log \frac{{r_3}}{2r_1}}$, and $C$ depends only on $d$, $\mu$, $R$, $T$ and $(\tau,\lambda)$, and $\tilde{\Omega}_{{r_3},t_0}=E_{r_3} \times (t_0-{r_3}^2,t_0)$ is a subdomain of $B_{R}\times (-T,T)$ with $R$ and $T$ fixed.
\end{thm}
\begin{rmk}1.It should be noticed that the unique continuation property of the solution $u$ to the equation $\partial_t u-\operatorname{div}(A(x,t)\nabla u)+Vu=0$ has been obtained under the assumption that $A\in C^{2,1}$ and $V\in L^{\frac{d+2}{2}}_{loc}(dxdt)$ in \cite{1990A}. To some extend, we have extended this result to parabolic equation in homogenization. We may refer readers to \cite{escauriaza2000carleman,escauriaza2004unique,escauriaza2001carleman,1990A,koch2009carleman} and their references therein for more results about the nonzero potential.

2. The method used in Theorem 1.4 can easily apply to the equation $-\operatorname{div}\left(A(x/\varepsilon)\nabla u_\varepsilon\right)+V^\varepsilon u_\varepsilon=0$ with suitable $V^\varepsilon$.
\end{rmk}
\section{Preliminaries}

Let $\mathcal{L}_\varepsilon =-\operatorname{div}\left(A_\varepsilon(x,t)\nabla \right)$, where $A_\varepsilon(x,t)=A(x/\varepsilon,t/\varepsilon^2)$.
Assume that $A(y,s)$ is 1-periodic in $(y,s)$ and satisfies the ellipticity condition $(1.5)$. For $1\leq j\leq d$, the corrector $\chi_j=\chi_j(y,s)$ is defined as the weak solution to the following cell problem:
\begin{equation}\left\{\begin{array}{l}
\left(\partial_{s}+\mathcal{L}_{1}\right)\left(\chi_{j}\right)=-\mathcal{L}_{1}\left(y_j\right) \quad \text { in } Y, \\
\chi_{j}=\chi_{j}^{\beta}(y, s) \text { is } 1 \text { -periodic in }(y, s), \\
\int_{Y} \chi_{j}^{\beta}=0,
\end{array}\right.\end{equation}where $Y=[0,1)^{d+1}.$ Note that

\begin{equation}
\left(\partial_{s}+\mathcal{L}_{1}\right)\left(\chi_{j}+y_j\right)=0\ \text{in }\mathbb{R}^{d+1}.
\end{equation}By the rescaling property of $\partial_t+\mathcal{L}_\varepsilon$, we obtain that

\begin{equation}
\left(\partial_{t}+\mathcal{L}_{\varepsilon}\right)\left(\varepsilon\chi_{j}(x/\varepsilon,t/\varepsilon^2)+y_j\right)=0\ \text{in }\mathbb{R}^{d+1}.
\end{equation}Moreover, if $A=A(y,s)$ is H\"{o}lder continuous in $(y,s)$, then by standard regularity for $\partial_s+\mathcal{L}_1$, $\nabla \chi_j(y,s)$ is H\"{o}lder continuous in $(y,s)$, thus $\nabla \chi_j(y,s)$ is bounded.\\

Let $\widehat{A}=(\widehat{a}_{ij})$, where $1\leq i,j\leq d$, and
\begin{equation}
\widehat{a}_{ij}=\fint_Y\left(a_{ij}+a_{ik}\frac{\partial \chi_j}{\partial y_k}\right)dyds.
\end{equation}
It is known that the constant matrix $\widehat{A}$ satisfies the ellipticity condition,
\begin{equation*}
\mu |\xi|^2\leq \widehat{a}_{ij}\xi_i\xi_j,\quad\quad \text{for any }\xi\in\mathbb{R}^d,
\end{equation*}and $|\widehat{a}_{ij}|\leq \mu_1$, where $\mu_1$ depends only on $d$ and $\mu$ \cite{bensoussan2011asymptotic}. It is also true or easy to verify  that $(\widehat{a_{ij}})$ is symmetric if $(a_{ij})$ is symmetric. Denote
$$\mathcal{L}_0=-\operatorname{div}(\widehat{A}\nabla).$$

Then $\partial_t+\mathcal{L}_0$ is the homogenized operator for the family of parabolic operators $\partial_t+\mathcal{L}_\varepsilon$, with $1>\varepsilon>0$. Since $\widehat{A}$ is symmetric and positive definite, there exists a $d\times d$ matrix $S$ with $\det(S)>0$ such that $S\widehat{A}S^T=I_{d\times d}$. Note that
$\widehat{A}^{-1}=S^TS$ and
\begin{equation}
\langle \widehat{A}^{-1}x,x\rangle=|Sx|^2.
\end{equation}
We introduce a family of ellipsoids as
\begin{equation}
E_r(\widehat{A})=\{x\in\mathbb{R}^n:\langle\widehat{A}^{-1}x,x\rangle< r^2\}.
\end{equation} It is easy to see that
\begin{equation}B_{\sqrt{\mu}r}(0)\subset E_r(0)\subset B_{\sqrt{\mu_1}r}(0).\end{equation}We will write $E_r(\widehat{A})$ as $E_r$ if the context is understood.

To move forward, let $\Gamma_\varepsilon(x,t;y,s)$ and $\Gamma_0(x,t;y,s)$ denote the fundamental solutions for the parabolic operators $\partial_t+\mathcal{L}_\varepsilon$, with $1>\varepsilon>0$ and the homogenized operator $\partial_t+\mathcal{L}_0$, respectively. Moreover, it is easy to see that
\begin{equation}
\Gamma_0(x,t;y,s)=\frac{1}{\left(2\sqrt{\pi}\right)^d}\left(t-s\right)^{-d/2}|S|\exp\left\{-\frac{|Sx-Sy|^2}{4(t-s)}\right\},
\end{equation}
for any $x,y\in\mathbb{R}^d$ and $-\infty<s<t<\infty$ with the matrix $S$ defined in $(2.5)$.\\

The following lemmas state the asymptotic behaviors of $\Gamma_\varepsilon(x,t;y,s)$ with $1>\varepsilon>0$, whose proof could be found in \cite{geng2020asymptotic}.

\begin{lemma}Suppose that the coefficient matrix $A$ satisfies the assumptions $(1.5)$ and $(1.6)$, then
\begin{equation}\left|\Gamma_{\varepsilon}(x, t ; y, s)-\Gamma_{0}(x, t ; y, s)\right| \leq \frac{C \varepsilon}{(t-s)^{\frac{d+1}{2}}} \exp \left\{-\frac{\kappa|x-y|^{2}}{t-s}\right\}\end{equation}for any $x,y\in\mathbb{R}^d$  and $-\infty<s<t<\infty$, where $\kappa>0$ depends only on $\mu$. The constant C depends only on $d$ and $\mu$.
\end{lemma}

The next lemma states the asymptotic behaviors of $\nabla_x\Gamma_\varepsilon(x,t;y,s)$ and $\nabla_y\Gamma_\varepsilon(x,t;y,s)$.
\begin{lemma}Suppose that the coefficient matrix $A$ satisfies the assumptions $(1.5)$, $(1.6)$ and $(1.7)$, then
\begin{equation}\begin{aligned}
\left|\nabla_{x} \Gamma_{\varepsilon}(x, t ; y, s)-\left(I+\nabla \chi\left(x / \varepsilon, t / \varepsilon^{2}\right)\right) \nabla_{x} \Gamma_{0}(x, t ; y, s)\right| \\
\quad \leq \frac{C \varepsilon}{(t-s)^{\frac{d+2}{2}}} \log \left(2+\varepsilon^{-1}|t-s|^{1 / 2}\right) \exp \left\{-\frac{\kappa|x-y|^{2}}{t-s}\right\}
\end{aligned}\end{equation}
for any $x,y\in\mathbb{R}^d$  and $-\infty<s<t<\infty$, where $\kappa>0$ depends only on $\mu$. The constant C depends only on $d$, $\mu$ and $(\tau,\lambda)$ in $(1.7)$. Similarly, there holds
\begin{equation}\begin{aligned}
\left|\nabla_{y} \Gamma_{\varepsilon}(x, t ; y, s)-\left(I+\nabla \widetilde{\chi}\left(y / \varepsilon,-s / \varepsilon^{2}\right)\right) \nabla_{y} \Gamma_{0}(x, t ; y, s)\right| \\
\quad \leq \frac{C \varepsilon}{(t-s)^{\frac{d+2}{2}}} \log \left(2+\varepsilon^{-1}|t-s|^{1 / 2}\right) \exp \left\{-\frac{\kappa|x-y|^{2}}{t-s}\right\},
\end{aligned}\end{equation}where $\tilde{\chi}(y,s)$ denote the correctors for $\partial_t+\tilde{\mathcal{L}}_\varepsilon$ with $\tilde{\mathcal{L}}_\varepsilon=-\operatorname{div}\left(A(x/\varepsilon,-t/\varepsilon^2)\nabla\right)$.
\end{lemma}
With the summation convention this means that for $1\leq i,j\leq d$,
\begin{equation}
\left|\frac{\partial\Gamma_{\varepsilon}(x, t ; y, s)}{\partial x_i}-\frac{\partial\Gamma_{0}(x, t ; y, s)}{\partial x_i}-\frac{\partial\chi_j\left(x / \varepsilon, t / \varepsilon^{2}\right)}{\partial x_i}\frac{ \partial \Gamma_{0}(x, t ; y, s)}{\partial x_j}\right|
\end{equation} is bounded by the RHS of $(2.10)$. And the similar result holds for $\nabla_{y} \Gamma_{\varepsilon}(x, t ; y, s)$.

The next lemma will be frequently used in the proof of Theorem 1.1.
\begin{lemma}
Let $u_\varepsilon$ be a weak solution of $\partial_t u_\varepsilon+\mathcal{L}_\varepsilon u_\varepsilon=0$ in $B_{R}\times (-T,T)$, then
\begin{equation}
\int_{t-r_3^2}^t\int_{E_{4{r_3}/5}\backslash E_{3{r_3}/4}}|\nabla u_\varepsilon|^2(x,s)dxds\leq Cr_3^d||u_\varepsilon||^2_{L^\infty\left(E_{r_3}\times (t-r_3^2,t)\right)},
\end{equation} where $C$ depends only on $\mu$ and $d$, and $E_{r_3}\times (t-r_3^2,t)$ is a subdomain of $B_{R}\times (-T,T)$.
\end{lemma}
\begin{proof}
The proof is standard. Choosing a cut-off function $\varphi\in [0,1]$ such that $\varphi(x)=1$ if $x\in E_{4{r_3}/5}\backslash E_{3{r_3}/4}$, and $\varphi(x)=0$ if $x\in E_{{r_3}/2}\cup \{\mathbb{R}^d\backslash E_{r_3}\}$ together with $|\nabla \varphi|\leq C/{r_3}$, then multiplying the equation $\partial_t u_\varepsilon+\mathcal{L}_\varepsilon u_\varepsilon=0$ by $\varphi^2u_\varepsilon$ and integrating the resulting equation over $B_{R}\times (t-r_3^2,t)$ leads to

\begin{equation}\begin{aligned}
&\int_{B_{R}}\varphi^2u_\varepsilon^2(x,t)dx+\int_{t-r_3^2}^t\int_{B_{R}}\varphi^2|\nabla u_\varepsilon|^2(x,s)dxds\\
\leq& C\int_{B_{R}}\varphi^2u_\varepsilon^2(x,t-{r_3}^2)dx+C\int_{t-{r_3}^2}^t\int_{B_{R}}|\nabla \varphi|^2 u_\varepsilon^2(x,s)dxds\\
\leq & Cr_3^d||u_\varepsilon||^2_{L^\infty\left(E_{r_3}\times (t-r_3^2,t)\right)}.
\end{aligned}\end{equation} Thus we have completed the proof of $(2.13)$ after noting the choice of $\varphi$.
\end{proof}

\section{Approximate two-sphere one-cylinder inequality}
Following \cite{guadie2014on}, we are going to apply the Lagrange interpolation method to obtain the approximate two-sphere one-cylinder inequality. Actually, the similar method in \cite{guadie2014on} has been used by the author in \cite{kenig2019propagation} to obtain the approximate three-ball inequality in elliptic periodic homogenization. First, let us briefly review the standard Lagrange interpolation method in numerical analysis. Set
\begin{equation}
\Phi_m(z)=(z-p_1)(z-p_2)\cdots(z-p_m)
\end{equation}for $z,p_j\in \mathcal{C}$ with $j=1,\cdots,m$. Let $\mathcal{D}$ ba a simply connected open domain in the complex plane
$\mathcal{C}$ that contains the nodes $\tilde{p},p_1,\cdots,p_m$. Assume that $f$ is an analytic function without poles in the closure of $\mathcal{D}$.
By well-known calculations, it holds that

\begin{equation}\frac{1}{z-\tilde{p}}=\sum_{j=1}^{m} \frac{\Phi_{j-1}(\tilde{p})}{\Phi_{j}(z)}+\frac{\Phi_{m}(\tilde{p})}{(z-\tilde{p}) \Phi_{m}(z)}.\end{equation}
Multiplying the last identify by $\frac{1}{2\pi i}f(z)$ and integrating along the boundary of $\mathcal{D}$ leads to
\begin{equation}
\frac{1}{2 \pi i} \int_{\partial \mathcal{D}} \frac{f(z)}{z-\tilde{p}} d z=\sum_{j=1}^{m} \frac{\Phi_{j-1}(\tilde{p})}{2 \pi i} \int_{\partial \mathcal{D}} \frac{f(z)}{\Phi_{j}(z)} d z+\left(R_{m} f\right)(\tilde{p}),\end{equation}
where

\begin{equation}
\left(R_{m} f\right)(\tilde{p})=\frac{1}{2 \pi i} \int_{\partial \mathcal{D}} \frac{\Phi_{m}(\tilde{p}) f(z)}{(z-\tilde{p}) \Phi_{m}(z)} d z.
\end{equation}
By the residue theorem, there holds that
\begin{equation}\begin{aligned}
\left({R}_{m} f\right)(\tilde{p}) &=\sum_{j=1}^{m} \frac{\Phi_{m}(\tilde{p})}{\left(p_{j}-\tilde{p}\right) \Phi_{m}^{\prime}\left(p_{j}\right)} f\left(p_{j}\right)+f(\tilde{p}) \\
&=-\sum_{j=1}^{m} \prod_{i \neq j}^{m} \frac{\tilde{p}-p_{i}}{p_{j}-p_{i}} f\left(p_{j}\right)+f(\tilde{p}),
\end{aligned}\end{equation}
where $\left(R_{m} f\right)(\tilde{p})$ is called the interpolation error. See chapter 4 in \cite{Bjorck2008Numerical} for more information.

In order to obtain the approximate two-sphere one-cylinder inequality for the solution in $(1.4)$, we consider the Lagrange interpolation for
$f(h)=\Gamma_0(hx_0\frac{r_1}{r_2},t_0;y,s)$, where $0<r_1<r_2<{{r_3}}/{12}<R/8$ and $(x_0,t_0)$ is a fixed point such that
$\sqrt{\langle\widehat{A}^{-1}x_0,x_0\rangle}=|Sx_0|< r_2$. In view of $(3.5)$, we need to estimate the error term $(R_m\Gamma_0)(x_0,t_0;y,s)$ of the
approximation.
Following the idea in \cite{kenig2019propagation}, we choose points $x_i=h_ix_0\frac{r_1}{r_2}$ on the segment $[0,x_0\frac{r_1}{r_2}]$ with $h_i\in (0,1)$, then $x_i\in E_{r_1}$, $i=1,\cdots,m$. Select $p_i=h_i$
in the definition of $\Phi_m$ in $(3.1)$ and $\tilde{p}=r_2/r_1$. Define
\begin{equation}
c_i=\prod_{j\neq i}^{m}\frac{r_2r_1^{-1}-h_j}{h_i-h_j}.
\end{equation}
Since $0<h_i<1$, direct computation shows that
\begin{equation}
|c_i|\leq \frac{\left(r_2r_1^{-1}\right)^{m-1}}{|\Phi_m'(h_i)|}.
\end{equation}
To estimate $|c_i|$, we choose $h_i$ to be the Chebyshev nodes, which means, $h_i=\cos\left(\frac{(2i-1)\pi}{2m}\right)$, $i=1,\cdots,m$. Then we can write
$$\Phi_m(h)=2^{1-m}T_m(h),$$ where $T_m$ is the Chebyshev polynomial of the first kind. There also holds that
\begin{equation}
\Phi_m'(h)=m2^{1-m}U_{m-1}(t),
\end{equation}
where $U_{m-1}$ is the Chebyshev polynomial of the second kind. See e.g. section 3.2.3 in \cite{Bjorck2008Numerical}. At each $h_i$, there hold

\begin{equation}
U_{m-1}(h_i)=U_{m-1}\left(\cos\left(\frac{(2i-1)\pi}{2m}\right)\right)=\frac{\sin{\frac{(2i-1)\pi}{2}}}{\sin{\frac{(2i-1)\pi}{2m}}}
=\frac{(-1)^{i-1}}{\sin\frac{(2i-1)\pi}{2m}}.
\end{equation}
According to $(3.8)$ and $(3.9)$, there holds

\begin{equation}
|\Phi_m'(h_i)|\geq m 2^{1-m}.
\end{equation}
Therefore, by $(3.7)$, we have
\begin{equation}
|c_i|\leq (m)^{-1}\left(\frac{2r_2}{r_1}\right)^{m-1} \text{ for }i=1,\cdots,m.
\end{equation}

To estimate the error term $(R_m\Gamma_0)(x_0,t_0;y,s)$, we do an analytic extension of the function $f(h)=\Gamma_0(hx_0\frac{r_1}{r_2},t_0;y,s)$
to the disc $\mathcal{D}_{\frac{{r_3}}{3r_1}}$ of radius $\frac{{r_3}}{3r_1}$ centered at the origin in the complex plane $\mathcal{C}$. According to $(2.8)$, we have
\begin{equation}
f(z)=\frac{1}{\left(2\sqrt{\pi}\right)^d}\left(t-s\right)^{-d/2}|S|\exp\left\{-\frac{(z\frac{r_1}{r_2}Sx_0-Sy)^2}{4(t-s)}\right\},
\end{equation} where $(z\frac{r_1}{r_2}Sx_0-Sy)^2=(z\frac{r_1}{r_2}Sx_0-Sy)\cdot (z\frac{r_1}{r_2}Sx_0-Sy)=\sum_{i=1}^d(z\frac{r_1}{r_2}(Sx_0)_i-(Sy)_i)^2$.
Note that $|z\frac{r_1}{r_2}Sx_0|\leq \frac{{r_3}}{3}$ in the disc $\mathcal{D}_{\frac{{r_3}}{3r_1}}$, then  there holds

\begin{equation}
|f(z)|\leq \tilde{C}\left(t-s\right)^{-d/2}\exp\left\{-\frac{Cr_3^2}{t-s}\right\}\quad \text{for } y\in E_{4{r_3}/5}\backslash E_{3{r_3}/4},
\end{equation}where $C$ and $\tilde{C}$ depend only on $d$.\\

Similarly, with the notations above, consider the Lagrange interpolation for
$g(h)=\nabla_y\Gamma_0(hx_0\frac{r_1}{r_2},t_0;y,s)$, and we do an analytic extension of the $g(h)$
to the disc $\mathcal{D}_{\frac{{r_3}}{3r_1}}$. Then according to $(2.8)$ again, there holds

\begin{equation}
g(z)=\frac{S\left(z\frac{r_1}{r_2}Sx_0-Sy\right)}{2\left(2\sqrt{\pi}\right)^d}\left(t-s\right)^{-d/2-1}|S|\exp\left\{-\frac{(z\frac{r_1}{r_2}Sx_0-Sy)^2}{4(t-s)}\right\},
\end{equation}where $S(z\frac{r_1}{r_2}Sx_0-Sy)$ is a vector with $S_{ik}\cdot(z\frac{r_1}{r_2}S_{ij}x_{0,j}-S_{ij}y_j)$ being its $k$-th position. Note that $|z\frac{r_1}{r_2}Sx_0|\leq \frac{{r_3}}{3}$ in the disc $\mathcal{D}_{\frac{{r_3}}{3r_1}}$, then we have

\begin{equation}
|g(z)|\leq \tilde{C}{r_3}\left(t-s\right)^{-d/2-1}\exp\left\{-\frac{Cr_3^2}{t-s}\right\}\quad \text{for } y\in E_{4{r_3}/5}\backslash E_{3{r_3}/4},
\end{equation}where $C$ and $\tilde{C}$ depend only on $d$.\\

The following lemma gives the interpolation error terms $(R_m(\nabla_y \Gamma_0))(x,t;y,s)$  and $(R_m\Gamma_0)(x,t;y,s)$ for $\nabla_y\Gamma_0(x,t;y,s)$ and $\Gamma_0(x,t;y,s)$, respectively.

\begin{lemma}
If $x_0\in E_{r_2}$ and $y\in E_{4{r_3}/5}\backslash E_{3{r_3}/4} $ with $0<r_1<r_2<{r_3}/12<R/8$ and $-\infty<s<t<\infty$, then there hold
\begin{equation}
|(R_m\Gamma_0)(x_0,t;y,s)|\leq \frac{\tilde{C}4^mr_2^m}{{r_3}^m}(t-s)^{-d/2}\exp\left\{-\frac{Cr_3^2}{t-s}\right\},
\end{equation}and

\begin{equation}
|(R_m(\nabla_y\Gamma_0))(x_0,t;y,s)|\leq \frac{\tilde{C}4^mr_2^m}{{r_3}^{m-1}}(t-s)^{-d/2-1}\exp\left\{-\frac{Cr_3^2}{t-s}\right\},
\end{equation}where $C$ and $\tilde{C}$ depend only on $d$.
\end{lemma}

\begin{proof}
First, to see $(3.16)$. According to $(3.1)$ and noting that $p_i=h_i\in(0,1)$ with $i=1,\cdots,m$, it is easy to see that
\begin{equation}
|\Phi_m(z)|\geq \left(\left(\frac{r_3}{3r_1}\right)-1\right)^m \text{ on the circle } |z|=\frac{r_3}{3r_1}
\end{equation}and
\begin{equation}
|\Phi_m(r_2/r_1)|\leq (r_2/r_1)^m.
\end{equation}
 In view of $(3.4)$-$(3.6)$ and $(3.18)$-$(3.19)$, we have
\begin{equation}\begin{aligned}
\left|(R_m\Gamma_0)(x_0,t;y,s)\right| &=|\Gamma_0(x_0,t;y,s)-\sum_{i=1}^{m} c_{i} \Gamma_0(x_i,t;y,s)| \\
&=|f\left(r_2/r_1\right)-\sum_{i=1}^{m} c_{i} f\left(h_{i}\right)| \\
&=|\frac{1}{2 \pi i} \int_{|z|=\frac{{r_3}}{3 r_{1}}} \frac{\Phi_{m}\left(r_{2} r_{1}^{-1}\right) f(z)}{\left(z-r_{2} r_{1}^{-1}\right) \Phi_{m}(z)} d z| \\
& \leq C \frac{(r_2/r_1)^m(3r_1)^m}{({r_3}-3r_2)({r_3}-3r_1)^m}\cdot {r_3}(t-s)^{-d/2}\exp\left\{-\frac{Cr_3^2}{t-s}\right\}\\
& \leq \frac{\tilde{C}4^mr_2^m}{{r_3}^m}(t-s)^{-d/2}\exp\left\{-\frac{Cr_3^2}{t-s}\right\},
\end{aligned}\end{equation}where we have used estimate $(3.13)$, the assumption that $0<r_1<r_2<{r_3}/12<R/8$ in the last inequality, and the constants $\tilde{C}$ and $C$ in the last inequality depend on $d$.\\

Similarly, for $(R_m(\nabla_y) \Gamma_0(x,t;y,s)$, there holds
\begin{equation}\begin{aligned}
\left|(R_m(\nabla_y\Gamma_0))(x_0,t;y,s)\right| &=|\nabla_y\Gamma_0(x_0,t;y,s)-\sum_{i=1}^{m} c_{i} \nabla_y\Gamma_0(x_i,t;y,s)| \\
&=|g\left(r_2/r_1\right)-\sum_{i=1}^{m} c_{i} g\left(h_{i}\right)| \\
&=|\frac{1}{2 \pi i} \int_{|z|=\frac{{r_3}}{2 r_{1}}} \frac{\Phi_{m}\left(r_{2} r_{1}^{-1}\right) g(z)}{\left(z-r_{2} r_{1}^{-1}\right) \Phi_{m}(z)} d z| \\
& \leq \frac{\tilde{C}4^mr_2^m}{{r_3}^{m-1}}(t-s)^{-d/2-1}\exp\left\{-\frac{Cr_3^2}{t-s}\right\},
\end{aligned}\end{equation}where we have used estimate $(3.15)$ instead of $(3.13)$, compared to $(3.20)$, and the assumption that $0<r_1<r_2<{r_3}/12<R/8$ in the last inequality, and the constants $\tilde{C}$ and $C$ in the last inequality depend on $d$. Thus we have completed the proof of Lemma 3.1.
\end{proof}
To continue the proof of Theorem 1.1, since $u_\varepsilon$ satisfies
\begin{equation}
\partial_t u_\varepsilon-\operatorname{div}\left(A_\varepsilon \nabla u_\varepsilon\right)=0 \quad \text{in }B_{{R}}\times (-T,T),
\end{equation}then simple computation shows that
\begin{equation}\begin{aligned}
&\partial_t \left(\eta u_\varepsilon\right)-\operatorname{div}\left(A_\varepsilon \nabla \left(\eta u_\varepsilon\right)\right)\\
&=-\operatorname{div}\left[\left( A_\varepsilon\cdot \nabla \eta\right) u_\varepsilon \right]-A_\varepsilon \nabla u_\varepsilon \nabla \eta +u_\varepsilon \partial_t \eta\\
&=:\tilde{f}(x,t)\quad\quad \text{ in }\mathbb{R}^d\times (-\infty,T),
\end{aligned}\end{equation}where $\eta\in[0,1]$ is a cut-off function such that $\eta=\eta(x,t)=1$ if $(x,t)\in E_{3{r_3}/4} \times (t_0-r_3^2/2,t_0)$, and $\eta=0$ if $(x,t)\notin
E_{4{r_3}/5} \times \{(t_0-3r_3^2/4,t_0+r_3^2/4)$ for some fixed $t_0$ with $|\nabla \eta|\leq C/{r_3}$ and $|\partial_t \eta|\leq C/r_3^2$. Then
\begin{equation}\begin{aligned}
\left(\eta u_\varepsilon\right)(x_0,t_0)&=\int_{t_0-r_3^2}^{t_0}\int_{\mathbb{R}^d}\Gamma_\varepsilon(x_0,t_0;y,s)\tilde{f}(y,s)dyds\\
&=\int_{t_0-r_3^2}^{t_0}\int_{\mathbb{R}^d}\nabla_y\Gamma_\varepsilon(x_0,t_0;y,s)\left( A_\varepsilon\cdot \nabla \eta\right) u_\varepsilon dyds \\
&\quad+\int_{t_0-r_3^2}^{t_0}\int_{\mathbb{R}^d}\Gamma_\varepsilon(x_0,t_0;y,s)\left(u_\varepsilon \partial_s \eta- A_\varepsilon \nabla u_\varepsilon \nabla \eta\right)dyds\\
&=:I_1+I_2,
\end{aligned}\end{equation}where $x_0\in E_{r_2} $ is a fixed point.

The summation convention that repeated indices are summed is used in the rest of this section.

It is easy to see that
\begin{equation}\begin{aligned}
I_1=&\int_{t_0-r_3^2}^{t_0}\int_{\mathbb{R}^d}c_i\nabla_y\Gamma_\varepsilon(x_i,t_0;y,s)(A_\varepsilon\cdot \nabla \eta) u_\varepsilon dyds\\
&+\int_{t_0-r_3^2}^{t_0}\int_{\mathbb{R}^d}\left[\nabla_y\Gamma_\varepsilon(x_0,t_0;y,s)-\left(I+\nabla
\widetilde{\chi}_\varepsilon\right)\nabla_y\Gamma_0(x_0,t_0;y,s)\right]\left( A_\varepsilon\cdot \nabla \eta\right) u_\varepsilon dyds\\
&+\int_{t_0-r_3^2}^{t_0}\int_{\mathbb{R}^d}\left(I+\nabla\widetilde{\chi}_\varepsilon\right)\left[\nabla_y\Gamma_0(x_0,t_0;y,s)
-c_i\nabla_y\Gamma_0(x_i,t_0;y,s)\right]\left( A_\varepsilon\cdot \nabla \eta\right) u_\varepsilon dyds\\
&+\int_{t_0-r_3^2}^{t_0}\int_{\mathbb{R}^d}c_i\left[\left(I+\nabla\widetilde{\chi}_\varepsilon\right)\nabla_y\Gamma_0(x_i,t_0;y,s)
-\nabla_y\Gamma_\varepsilon(x_i,t_0;y,s)\right]\left( A_\varepsilon\cdot \nabla \eta\right) u_\varepsilon dyds\\
=&:M_0+M_1+M_2+M_3,
\end{aligned}\end{equation} where $c_i$ are defined in $(3.6)$ and $x_i=h_ix_0\frac{r_1}{r_2}$ on the segment $[0,x_0\frac{r_1}{r_2}]$ with $h_i\in (0,1)$, $i=1,\cdots,m$ and $\nabla \widetilde{\chi}_\varepsilon=\nabla \widetilde{\chi}\left(y / \varepsilon,-s / \varepsilon^{2}\right)$. Moreover, it is easy to see that $x_i\in E_{r_1} $, $i=1,\cdots,m$. Similarly, we have
\begin{equation}\begin{aligned}
I_2=&\int_{t_0-r_3^2}^{t_0}\int_{\mathbb{R}^d}c_i\Gamma_\varepsilon(x_i,t_0;y,s)\left(u_\varepsilon \partial_s \eta- A_\varepsilon \nabla u_\varepsilon \nabla \eta\right)dyds\\
&+\int_{t_0-r_3^2}^{t_0}\int_{\mathbb{R}^d}\left[\Gamma_\varepsilon(x_0,t_0;y,s)-\Gamma_0(x_0,t_0;y,s)\right]\left(u_\varepsilon \partial_s \eta- A_\varepsilon \nabla u_\varepsilon \nabla \eta\right) dyds\\
&+\int_{t_0-r_3^2}^{t_0}\int_{\mathbb{R}^d}\left[\Gamma_0(x_0,t_0;y,s)-c_i\Gamma_0(x_i,t_0;y,s)\right]\left(u_\varepsilon \partial_s \eta- A_\varepsilon \nabla u_\varepsilon \nabla \eta\right) dyds\\
&+\int_{t_0-r_3^2}^{t_0}\int_{\mathbb{R}^d}c_i\left[\Gamma_0(x_i,t_0;y,s)-\Gamma_\varepsilon(x_i,t_0;y,s)\right]\left(u_\varepsilon \partial_s \eta- A_\varepsilon \nabla u_\varepsilon \nabla \eta\right) dyds\\
=&\tilde{M}_0+\sum_{i=4}^{i=9}M_i,
\end{aligned}\end{equation} with the same $c_i$, $x_i$ and $x_0$ as in $(3.25)$.  Clearly, it follows from the representation formula $(3.24)$ that
\begin{equation}M_0+\tilde{M}_0=c_i\left(\eta u_\varepsilon\right)(x_i,t_0).\end{equation}

Before we continue, we give some notations first. Denote $\tilde{E}$ and $\tilde{\Omega}_{r_3,t_0}$ by $E_{4r_3/5} \setminus E_{3r_3/4}$ and $E_{r_3} \times (t_0-r_3^2,t_0)$, respectively. Next, we need to estimate $M_1$-$M_9$ term by term. In view of $(3.25)$, we have

\begin{equation}\begin{aligned}
|M_1|&\leq C\int_{t_0-r_3^2}^{t_0}\int_{\mathbb{R}^d}\left|\nabla_y\Gamma_\varepsilon(x_0,t_0;y,s)
-\left(I+\nabla\widetilde{\chi}_\varepsilon\right)\nabla_y\Gamma_0(x_0,t_0;y,s)\right||\nabla \eta|| u_\varepsilon| dyds\\
&\leq C\varepsilon\int_{t_0-r_3^2}^{t_0}\int_{\mathbb{R}^d}\frac{\log \left(2+\varepsilon^{-1}|t_0-s|^{1 / 2}\right)}{(t_0-s)^{\frac{d+2}{2}}}  \exp \left\{-\frac{\kappa|x_0-y|^{2}}{t_0-s}\right\}|\nabla \eta|| u_\varepsilon| dyds\\
&\leq C\varepsilon\int_{t_0-r_3^2}^{t_0}\int_{\mathbb{R}^d}\frac{\log \left(2+\varepsilon^{-1}|t_0-s|^{1 / 2}\right)}{(t_0-s)^{\frac{d+2}{2}}}  \exp \left\{-\frac{C|Sx_0-Sy|^{2}}{t_0-s}\right\}|\nabla \eta|| u_\varepsilon| dyds\\
&\leq \frac{C\varepsilon}{r_3}\int_{t_0-r_3^2}^{t_0}\int_{\tilde{E}}\frac{\log \left(2+\varepsilon^{-1}|t_0-s|^{1 / 2}\right)}{(t_0-s)^{\frac{d+2}{2}}}  \exp \left\{-\frac{C|Sx_0-Sy|^{2}}{t_0-s}\right\}| u_\varepsilon| dyds\\
&\leq C\varepsilon r_3^{d-1}\int_{t_0-r_3^2}^{t_0}\frac{\log \left(2+\varepsilon^{-1}|t_0-s|^{1 / 2}\right)}{(t_0-s)^{\frac{d+2}{2}}}  \exp \left\{-\frac{Cr_3^2}{t_0-s}\right\}ds\cdot ||u_\varepsilon||_{L^\infty(\tilde{\Omega}_{r_3,t_0})}\\
&\leq C\varepsilon r_3^{-1}\int_1^\infty\log (2+\varepsilon^{-1}\tilde{s}^{-1/2}r_3){\tilde{s}}^{d/2-1}\exp \{-C\tilde{s}\}d\tilde{s}\cdot||u_\varepsilon||_{L^\infty(\tilde{\Omega}_{r_3,t_0})}\\
&\leq C\varepsilon r_3^{-1}\log(2+\varepsilon^{-1}r_3)||u_\varepsilon||_{L^\infty(\tilde{\Omega}_{r_3,t_0})},
\end{aligned}\end{equation}where we have used $(2.11)$ in Lemma 2.2 in the above inequality.

To estimate $M_2$, we first note that $\nabla \tilde{\chi}_\varepsilon$ is bounded, then according to Lemma 3.1, there holds
\begin{equation}\begin{aligned}
|M_2|&\leq\frac{C(4r_2)^m}{r_3^{m}}\int_{t_0-r_3^2}^{t_0}\int_{\tilde{E}}(t_0-s)^{-d/2-1}\exp\left\{-\frac{Cr_3^2}{t_0-s}\right\}dyds\cdot ||u_\varepsilon||_{L^\infty(\tilde{\Omega}_{r_3,t_0})}\\
&\leq \frac{C(4r_2)^m}{r_3^{m-d}}\int_{t_0-r_3^2}^{t_0}(t_0-s)^{-d/2-1}\exp\left\{-\frac{Cr_3^2}{t_0-s}\right\}ds\cdot ||u_\varepsilon||_{L^\infty(\tilde{\Omega}_{r_3,t_0})}\\
&\leq \frac{C(4r_2)^m}{r_3^{m}}\int_{1}^{\infty}{\tilde{s}}^{d/2-1}\exp\left\{-C\tilde{s}\right\}ds\cdot ||u_\varepsilon||_{L^\infty(\tilde{\Omega}_{r_3,t_0})}\\
&\leq \frac{C(4r_2)^m}{r_3^{m}}||u_\varepsilon||_{L^\infty(\tilde{\Omega}_{r_3,t_0})}.
\end{aligned}\end{equation}

As for $M_3$, noting that $x_i\in E_{r_1}$ with $i=1,\cdots,m$, then the estimate $(2.11)$ and $(3.11)$ yields that
\begin{equation}\begin{aligned}
|M_3|&\leq \frac{C\varepsilon}{r_3}\int_{t_0-r_3^2}^{t_0}\int_{\tilde{E}}\frac{\log \left(2+\varepsilon^{-1}|t_0-s|^{1 / 2}\right)}{(t_0-s)^{\frac{d+2}{2}}}|c_i| \exp \left\{-\frac{\kappa|x_i-y|^{2}}{t_0-s}\right\}dyds\cdot||u_\varepsilon||_{L^\infty(\tilde{\Omega}_{r_3,t_0})}\\
&\leq \frac{C\varepsilon}{r_3}\int_{t_0-r_3^2}^{t_0}\int_{\tilde{E}}\frac{\log \left(2+\varepsilon^{-1}|t_0-s|^{1 / 2}\right)}{(t_0-s)^{\frac{d+2}{2}}}|c_i| \exp \left\{-\frac{C|Sx_i-Sy|^{2}}{t_0-s}\right\}dyds\cdot||u_\varepsilon||_{L^\infty(\tilde{\Omega}_{r_3,t_0})}\\
&\leq C\varepsilon r_3^{d-1}\sum_{i}|c_i|\int_{t_0-r_3^2}^{t_0}\frac{\log \left(2+\varepsilon^{-1}|t_0-s|^{1 / 2}\right)}{(t_0-s)^{\frac{d+2}{2}}}\exp \left\{-\frac{Cr_3^2}{t_0-s}\right\}dyds\cdot||u_\varepsilon||_{L^\infty(\tilde{\Omega}_{r_3,t_0})}\\
&\leq \frac{C\varepsilon(2r_2)^{m-1}}{r_1^{m-1}r_3}\int_1^\infty\log (2+\varepsilon^{-1}\tilde{s}^{-1/2}r_3){\tilde{s}}^{d/2-1}\exp \{-C\tilde{s}\}d\tilde{s}\cdot||u_\varepsilon||_{L^\infty(\tilde{\Omega}_{r_3,t_0})}\\
&\leq \frac{C(2r_2)^{m-1}}{r_1^{m-1}}\varepsilon {r_3}^{-1}\log (2+\varepsilon^{-1}r_3)||u_\varepsilon||_{L^\infty(\tilde{\Omega}_{r_3,t_0})}.
\end{aligned}\end{equation}

Next, we give the estimate  of $I_2$ term by term  in $(3.26)$. In view of $(2.9)$, then we have
\begin{equation}\begin{aligned}
|M_4|&\leq C\varepsilon r_3^{d-2}\int_{t_0-r_3^2}^{t_0}(t_0-s)^{-\frac{d+1}{2}} \exp \left\{-\frac{Cr_3^2}{t_0-s}\right\}ds\cdot ||u_\varepsilon||_{L^\infty(\tilde{\Omega}_{r_3,t_0})}\\
&\leq C\varepsilon r_3^{d-2} \int_1^\infty {\tilde{s}}^{\frac{d-3}{2}}\exp (-C\tilde{s}) r_3^{-d+1}d\tilde{s}\cdot ||u_\varepsilon||_{L^\infty(\tilde{\Omega}_{r_3,t_0})}\\
&\leq C\varepsilon r_3^{-1}||u_\varepsilon||_{L^\infty(\tilde{\Omega}_{r_3,t_0})}.
\end{aligned}\end{equation}

As for $M_5$, we have
\begin{equation}\begin{aligned}
|M_5|&\leq C\left(\int_{t_0-r_3^2}^{t_0}\int_{\mathbb{R}^d}\left|\Gamma_\varepsilon(x_0,t_0;y,s)-\Gamma_0(x_0,t_0;y,s)\right|^2|\nabla \eta|dyds\right)^{1/2}\\
&\quad \times\left(\int_{t_0-r_3^2}^{t_0}\int_{\mathbb{R}^d}\left| \nabla u_\varepsilon \right|^2|\nabla \eta| dyds\right)^{1/2}\\
&\leq C\varepsilon r_3^{d/2-1}\left(\int_{t_0-r_3^2}^{t_0} (t_0-s)^{-d-1}\exp \left\{-\frac{Cr_3^2}{t_0-s}\right\}ds\right)^{1/2}\left(\int_{t_0-r_3^2}^{t_0}\int_{\tilde{E}} |\nabla u_\varepsilon|^2dyds\right)^{1/2}\\
&\leq C\varepsilon r_3^{-1}||u_\varepsilon||_{L^\infty(\tilde{\Omega}_{r_3,t_0})},
\end{aligned}\end{equation}where we have used $(2.9)$ in the second inequality and $(2.13)$ in the third inequality.

Due to $(3.16)$ and similar to the estimate of $M_2$, we have
\begin{equation}\begin{aligned}
|M_6|&\leq\frac{\tilde{C}(4r_2)^m}{r_3^{m-d+2}}\int_{t_0-r_3^2}^{t_0}(t_0-s)^{-d/2}\exp\left\{-\frac{Cr_3^2}{t_0-s}\right\}ds
\cdot||u_\varepsilon||_{L^\infty(\tilde{\Omega}_{r_3,t_0})}\\
&\leq \frac{C(4r_2)^m}{r_3^{m}}||u_\varepsilon||_{L^\infty(\tilde{\Omega}_{r_3,t_0})}.
\end{aligned}\end{equation}According to $(3.16)$ and $(2.13)$, then there holds

\begin{equation}\begin{aligned}
|M_7|&\leq C\left(\int_{t_0-r_3^2}^{t_0}\int_{\mathbb{R}^d}\left|\Gamma_0(x_0,t_0;y,s)-c_i\Gamma_0(x_i,t_0;y,s)\right|^2|\nabla \eta|dyds\right)^{1/2}\\
&\quad \times\left(\int_{t_0-r_3^2}^{t_0}\int_{\mathbb{R}^d}\left| \nabla u_\varepsilon \right|^2|\nabla \eta| dyds\right)^{1/2}\\
&\leq \frac{\tilde{C}(4r_2)^m}{r_3^{m+1-d}}\left(\int_{t_0-r_3^2}^{t_0}(t_0-s)^{-d}\exp\left\{-\frac{Cr_3^2}{t_0-s}\right\}ds\right)^{1/2}\cdot ||u_\varepsilon||_{L^\infty(\tilde{\Omega}_{r_3,t_0})}\\
&\leq\frac{C(4r_2)^m}{r_3^{m}}||u_\varepsilon||_{L^\infty(\tilde{\Omega}_{r_3,t_0})}.
\end{aligned}\end{equation}
Similarly, in view of $(2.9)$ and $(3.11)$, we have
\begin{equation}\begin{aligned}
|M_8|&\leq C\varepsilon r_3^{d-2}\sum_i|c_i|\int_{t_0-r_3^2}^{t_0}(t_0-s)^{-\frac{d+1}{2}} \exp \left\{-\frac{Cr_3^2}{t_0-s}\right\}ds
\cdot||u_\varepsilon||_{L^\infty(\tilde{\Omega}_{r_3,t_0})}\\
&\leq \frac{C\varepsilon (2r_2)^{m-1}}{r_3r_1^{m-1}}||u_\varepsilon||_{L^\infty(\tilde{\Omega}_{r_3,t_0})}.
\end{aligned}\end{equation}
Moreover, similar to the proof of $M_5$ and $(3.11)$, we have
\begin{equation}
|M_9|\leq \frac{C\varepsilon (2r_2)^{m-1}}{r_3r_1^{m-1}}||u_\varepsilon||_{L^\infty(\tilde{\Omega}_{r_3,t_0})}.
\end{equation}

Consequently, noting that $\sum_{i}|c_i|\leq (2r_2)^{m-1}/{r_1}^{m-1}$, then combining $(3.24)$-$(3.36)$ yields that
\begin{equation}\begin{aligned}
|u_\varepsilon(x_0,t_0)|\leq& \frac{(2r_2)^{m-1}}{{r_1}^{m-1}}\sup_{E_{r_1}}|u_\varepsilon(\cdot,t_0)|+\frac{C(4r_2)^m}{r_3^{m}}||u_\varepsilon||_{L^\infty(\tilde{\Omega}_{r_3,t_0})}\\
&+\frac{C(2r_2)^m}{r_1^m}\varepsilon r_3^{-1}\log (2+\varepsilon^{-1}r_3)||u_\varepsilon||_{L^\infty(\tilde{\Omega}_{r_3,t_0})},
\end{aligned}\end{equation}where $C$ does not depend on $m$, $r_1$, $r_2$ or $r_3$. Note that we choose the coefficient of the third term in the RHS of $(3.37)$ is $\frac{(2r_2)^m}{r_1^m}$ instead of $\frac{(2r_2)^{m-1}}{r_1^{m-1}}$, which can be done due to $0<r_1<r_2$, and that will simply the computation when minimizing the summation (of course, one could use $\frac{(2r_2)^{m-1}}{r_1^{m-1}}$ to obtain a more accurate conclusion). Since $x_0\in E_{r_2}$ is an arbitrary point, then it follows that

\begin{equation}\begin{aligned}
\sup_{E_{r_2}}|u_\varepsilon(\cdot,t_0)|\leq& C\left\{ \frac{(2r_2)^{m-1}}{{r_1}^{m-1}}\sup_{E_{r_1}}|u_\varepsilon(\cdot,t_0)|+\frac{(4r_2)^m}{r_3^{m}}\sup_{\tilde{\Omega}_{r_3,t_0}}|u_\varepsilon|\right.\\
&\quad\quad\left.+\frac{(2r_2)^m}{r_1^m}\varepsilon r_3^{-1}\log (2+\varepsilon^{-1}r_3)\sup_{\tilde{\Omega}_{r_3,t_0}}|u_\varepsilon|\right\}.
\end{aligned}\end{equation}

Now we need to minimize the summation of the terms in the RHS of $(3.38)$ by choosing the suitable integer value $m$. Actually, the similar proof can be found in \cite{kenig2019propagation}, we give it just for completeness. For simplicity, let
\begin{equation}
\sup_{E_{r_1}}|u_\varepsilon(\cdot,t_0)|=\delta,\quad \sup_{\tilde{\Omega}_{r_3,t_0}}|u_\varepsilon|=N.
\end{equation}First, choose $m$ such that
\begin{equation}
\delta\left(\frac{2r_2}{r_1}\right)^m=N\left(\frac{4r_2}{r_3}\right)^m,
\end{equation}which gives $$m=\frac{\log ({N/\delta})}{\log [{r_3/(2r_1)}]}.$$
Consequently, define
\begin{equation}
m_0= \left\lfloor\frac{\log ({N/\delta})}{\log [{r_3/(2r_1)}]}\right\rfloor+1,
\end{equation}where $\lfloor\cdot\rfloor$ denotes its integer part. We minimize the above terms by considering two cases.\\

\textbf{Case 1}. $\varepsilon r_3^{-1}\log (2+\varepsilon^{-1}r_3)\left(\frac{2r_2}{r_1}\right)^{m_0}\leq \left(\frac{4r_2}{r_3}\right)^{m_0}$.

In this case, let $m=m_0$ in $(3.38)$. Then the third term can be absorbed into the second one in the right hand side of $(3.38)$. Consequently, since $0<r_1<r_2<r_3/12$ and $\frac{\log ({N/\delta})}{\log [{r_3/(2r_1)}]}\leq m_0\leq \frac{\log ({N/\delta})}{\log [{r_3/(2r_1)}]}+1$,
it follows that
\begin{equation}\begin{aligned}
\sup _{E_{r_{2}}}\left|u_{\varepsilon}\right| & \leq C\left\{\delta\left(\frac{2 r_{2}}{r_{1}}\right)^{m_{0}-1}+N\left(\frac{4 r_{2}}{r_3}\right)^{m_{0}} \right\} \\
&\leq C\left\{\delta\left(\frac{2 r_{2}}{r_{1}}\right)^{\frac{\log ({N/\delta})}{\log [{r_3/(2r_1)}]}}+N\left(\frac{4 r_{2}}{r_3}\right)^{\frac{\log ({N/\delta})}{\log [{r_3/(2r_1)}]}} \right\}\\
& \leq C N^{1-\alpha} \delta^{\alpha},
\end{aligned}\end{equation}where
\begin{equation}
\alpha=\frac{\log {\frac{r_3}{4r_2}}}{\log\frac{r_3}{2r_1}}.
\end{equation}

\textbf{Case 2.} $\varepsilon r_3^{-1}\log (2+\varepsilon^{-1}r_3)\left(\frac{2r_2}{r_1}\right)^{m_0}> \left(\frac{4r_2}{r_3}\right)^{m_0}$.

In this case, from the definition of $m_0$, there holds that
\begin{equation}
\varepsilon r_3^{-1}\log (2+\varepsilon^{-1}r_3)>\frac{2\delta r_1}{Nr_3}.
\end{equation}That is,

\begin{equation}
\sup_{E_{r_1}}|u_\varepsilon(\cdot,t_0)|\leq \varepsilon r_3^{-1}\log (2+\varepsilon^{-1}r_3)\frac{r_3}{2r_1}\sup_{\tilde{\Omega}_{r_3,t_0}}|u_\varepsilon|.
\end{equation}Then, we choose $\widehat{m}$ such that

\begin{equation}\varepsilon r_3^{-1}\log (2+\varepsilon^{-1}r_3)\left(\frac{2r_2}{r_1}\right)^{\widehat{m}}=\left(\frac{4r_2}{r_3}\right)^{\widehat{m}},\end{equation}
which gives
$$\widehat{m}=\frac{\log[{\varepsilon r_3^{-1}\log (2+\varepsilon^{-1}r_3)}]}{\log \frac{2r_1}{r_3}}.$$
Therefore, we can choose
\begin{equation}
m_1=\left\lfloor\frac{\log[{\varepsilon r_3^{-1}\log (2+\varepsilon^{-1}r_3)}]}{\log \frac{2r_1}{r_3}}\right\rfloor+1.
\end{equation}Taking $m=m_1$ in $(3.38)$, then the second term can be absorbed into the third term in the RHS of $(3.38)$. In view of $(3.45)$, and noting
 $0<r_1<r_2<r_3/12$ and $\frac{\log[{\varepsilon r_3^{-1}\log (2+\varepsilon^{-1}r_3)}]}{\log \frac{2r_1}{r_3}}\leq m_1\leq \frac{\log[{\varepsilon r_3^{-1}\log (2+\varepsilon^{-1}r_3)}]}{\log \frac{2r_1}{r_3}}+1$, then we have

\begin{equation}\begin{aligned}
&\quad\sup_{E_{r_2}}|u_\varepsilon(x,t_0)|\\
 & \leq C\left\{\left(\frac{2 r_{2}}{r_{1}}\right)^{m_{1}-1} \varepsilon r_3^{-1}\log (2+\varepsilon^{-1}r_3) \frac{r_3}{r_{1}}N+\varepsilon r_3^{-1}\log (2+\varepsilon^{-1}r_3)\left(\frac{2 r_{2}}{r_{1}}\right)^{m_{1}} N\right\} \\
& \leq C \frac{ r_3}{r_{1}} \exp \left\{\frac{\log \frac{2 r_{2}}{r_{1}}\cdot \log \left[\varepsilon r_3^{-1}\log (2+\varepsilon^{-1}r_3)\right]}{\log \frac{2r_{1}}{r_3}}\right\} \varepsilon r_3^{-1}\log (2+\varepsilon^{-1}r_3) N \\
& \leq C \frac{r_3}{r_{1}}\left[\varepsilon r_3^{-1}\log (2+\varepsilon^{-1}r_3)\right]^{1+\frac{\log\frac{2r_2}{r_1}}{\log\frac{2r_1}{r_3}}}N \\
&=C \frac{r_3}{r_{1}}\left[\varepsilon r_3^{-1}\log (2+\varepsilon^{-1}r_3)\right]^{\alpha} N,
\end{aligned}\end{equation}
where
\begin{equation}
\alpha=\frac{\log {\frac{r_3}{4r_2}}}{\log\frac{r_3}{2r_1}}.
\end{equation}Notice that $0<\alpha<1$ according to the assumption of $r_1,r_2$ and $r_3$. Consequently, combining the two cases above yields the result of Theorem 1.1.  And Corollary 1.2 directly follows from Theorem 1.1 and the estimate $(2.7)$.
\section{Parabolic equation with potential in homogenization}
In this section, we give the proof of Theorem 1.4.
Denote $\mathcal{M}(V)=\int_{Z} V(z)dz$, and $v_\varepsilon =e^{\mathcal{M}(V)t}u_\varepsilon$. Then it is easy to see that $v_\varepsilon$ satisfying
\begin{equation}
\partial_t v_\varepsilon-\operatorname{div}\left(A(x/\varepsilon,t/\varepsilon^2)\nabla v_\varepsilon\right)=(\mathcal{M}(V)-V(x/\varepsilon))v_\varepsilon.
\end{equation}
Then simple computation shows that
\begin{equation}\begin{aligned}
&\partial_t \left(\eta v_\varepsilon\right)-\operatorname{div}\left(A_\varepsilon \nabla \left(\eta v_\varepsilon\right)\right)\\
&=-\operatorname{div}\left[\left( A_\varepsilon\cdot \nabla \eta\right) v_\varepsilon \right]-A_\varepsilon \nabla v_\varepsilon \nabla \eta +u_\varepsilon \partial_t \eta+(\mathcal{M}(V)-V(x/\varepsilon))v_\varepsilon \eta\\
&=:\tilde{g}(x,t)\quad\quad \text{ in }\mathbb{R}^d\times (-\infty,T),
\end{aligned}\end{equation}
where $\eta\in[0,1]$ is a cut-off function such that $\eta=\eta(x,t)=1$ if $(x,t)\in E_{3{r_3}/4} \times (t_0-r_3^2/2,t_0)$, and $\eta=0$ if $(x,t)\notin
E_{4{r_3}/5} \times (t_0-3r_3^2/4,t_0+r_3^2/4)$ for some fixed $t_0$ with $|\nabla \eta|\leq C/{r_3}$ and $|\partial_t
\eta|\leq C/r_3^2$.
Then
\begin{equation}\begin{aligned}
\left(\eta v_\varepsilon\right)(x_0,t_0)&=\int_{t_0-r_3^2}^{t_0}\int_{\mathbb{R}^d}\Gamma_\varepsilon(x_0,t_0;y,s)\tilde{g}(y,s)dyds\\
&=\int_{t_0-r_3^2}^{t_0}\int_{\mathbb{R}^d}\nabla_y\Gamma_\varepsilon(x_0,t_0;y,s)\left( A_\varepsilon\cdot \nabla \eta\right) v_\varepsilon dyds \\
&\quad+\int_{t_0-r_3^2}^{t_0}\int_{\mathbb{R}^d}\Gamma_\varepsilon(x_0,t_0;y,s)\left(u_\varepsilon \partial_s \eta- A_\varepsilon \nabla u_\varepsilon \nabla \eta\right)dyds\\
&\quad +\int_{t_0-r_3^2}^{t_0}\int_{\mathbb{R}^d}\Gamma_\varepsilon(x_0,t_0;y,s)(\mathcal{M}(V)-V(y/\varepsilon))v_\varepsilon \eta dyds\\
\end{aligned}\end{equation}where $x_0\in E_{r_2} $ is a fixed point. Noting that $\text{supp}((\mathcal{M}(V)-V(y/\varepsilon))v_\varepsilon \eta)\subset E_{r_3}\times (t_0-r_3^2,t_0+r_3^2)$ and in view the proof of $(3.16)$, we can't apply Lemma 3.1 to the last term on the RHS of $(4.3)$ to
estimate the following term
\begin{equation*}
\int_{t_0-r_3^2}^{t_0}\int_{\mathbb{R}^d}(\Gamma_0(x_0,t_0;y,s)-c_i\Gamma_0(x_i,t_0;y,s))(\mathcal{M}(V)-V(y/\varepsilon))v_\varepsilon \eta dyds.
\end{equation*}
However, thanks to the term $\mathcal{M}(V)-V(y/\varepsilon)$, which will give us the term $O(\varepsilon)$ after integrating by parts. In view of $(3.25)$-$(3.27)$, we need only to estimate the following term
\begin{equation}
I_3=:\int_{t_0-r_3^2}^{t_0}\int_{\mathbb{R}^d}(\Gamma_\varepsilon(x_0,t_0;y,s)-c_i\Gamma_\varepsilon(x_i,t_0;y,s))(\mathcal{M}(V)-V(y/\varepsilon))v_\varepsilon \eta dyds.
\end{equation}
Let $\psi(z)\in W^{2,\frac{d+2}{2}}(Z)$ and be 1-periodic, solving the following equation
\begin{equation}
\Delta_z\psi(z)=\mathcal{M}(V)-V(z) \text{ in $Z=[0,1)^d$, with }\int_Z\psi(z)dz=0.
\end{equation}
Then we have
\begin{equation}\begin{aligned}
I_3=&\varepsilon^2\int_{t_0-r_3^2}^{t_0}\int_{\mathbb{R}^d}\Delta_y \psi(y/\varepsilon)(\Gamma_\varepsilon(x_0,t_0;y,s)-c_i\Gamma_\varepsilon(x_i,t_0;y,s))v_\varepsilon \eta dyds\\
=&-\varepsilon\int_{t_0-r_3^2}^{t_0}\int_{\mathbb{R}^d}\nabla_z \psi(y/\varepsilon)\nabla_y(\Gamma_\varepsilon(x_0,t_0;y,s)-c_i\Gamma_\varepsilon(x_i,t_0;y,s))v_\varepsilon \eta dyds\\
&-\varepsilon\int_{t_0-r_3^2}^{t_0}\int_{\mathbb{R}^d}\nabla_z \psi(y/\varepsilon)(\Gamma_\varepsilon(x_0,t_0;y,s)-c_i\Gamma_\varepsilon(x_i,t_0;y,s))\nabla_y v_\varepsilon \eta dyds\\
&-\varepsilon\int_{t_0-r_3^2}^{t_0}\int_{\mathbb{R}^d}\nabla_z \psi(y/\varepsilon)(\Gamma_\varepsilon(x_0,t_0;y,s)-c_i\Gamma_\varepsilon(x_i,t_0;y,s))v_\varepsilon \nabla_y\eta dyds\\
=&I_4+I_5+I_6.
\end{aligned}\end{equation}
Similar to the estimate of $M_2$ in $(3.29)$, the term $I_6$ is easy to handle which we omit it here. In view of the definition of $\varphi$, $(3.11)$, and $0<r_1<r_2<r_3/12$, with $r_3\geq \varepsilon$, we have
\begin{equation}\begin{aligned}
|I_4|
\leq & C\varepsilon\int_{t_0-r_3^2}^{t_0}\int_{\mathbb{R}^d}|\nabla_z \psi^\varepsilon|\left(|\nabla_y(\Gamma_\varepsilon(x_0,t_0;y,s)|-|c_i\Gamma_\varepsilon(x_i,t_0;y,s)|\right)|v_\varepsilon| \eta dyds\\
\leq& C\varepsilon \sum_i|{c_i}|\sup_{\tilde{\Omega}_{r_3,t_0}}|v_\varepsilon|\int_{t_0-r_3^2}^{t_0}\int_{E_{r_3}}|\nabla_z \psi^\varepsilon|{(t_0-s)^{-\frac{1+d}{2}}}\exp\left\{-\frac{C|y|^2}{t_0-s}\right\} dyds\\
\leq&C\varepsilon \sum_i|{c_i}|||\nabla_z \psi^\varepsilon||_{L^{p_1'}(\tilde{\Omega}_{{r_3},t_0})}\sup_{\tilde{\Omega}_{r_3,t_0}}|v_\varepsilon|\\
&\quad\times\left(\int_{t_0-r_3^2}^{t_0}\int_{E_{\frac{r_3}{t_0-s}}}{(t-s)^{\frac{d}{2}(1-p_1)-\frac{p_1}{2}}}\exp\left\{-C|y|^2\right\} dyds\right)^{1/p_1}\\
\leq& C \varepsilon r_3 (\frac{2r_2}{r_1})^{m-1}\sup_{\tilde{\Omega}_{r_3,t_0}}|v_\varepsilon|,
\end{aligned}\end{equation}
with $1<p_1<\frac{d+1}{d}$ close to $\frac{d+1}{d}$ and $\psi^\varepsilon=\psi(y/\varepsilon)$,
where we have used $x_0-y\in E_{r_3}$ and $x_i-y\in E_{r_3}$ if $y\in E_{4r_3/5}$, $x_0\in E_{r_2}$ and $x_i\in E_{r_1}$ for $i=1,\cdots,m$, and  the size estimates $|\nabla_y \Gamma_\varepsilon(x,t;y,s)|\leq C(t-s)^{-\frac{d+1}{2}}\exp\left\{-\frac{\kappa|x-y|^2}{t-s}\right\}$ if $a_{ij}$ is H\"{o}lder continuous \cite{geng2020asymptotic}, as well as  the following inequality
\begin{equation}\begin{aligned}
\left(\int_{E_{r_3}} |\nabla_z \psi^\varepsilon|^{p'_1}dz\right)^{1/p'_1}&=\varepsilon^{d/p'_1}\left(\int_{E_{r_3/\varepsilon}} |\nabla_z \psi(z)|^{p'_1}dz\right)^{1/p'_1}\\
&\leq C r_3^{d/p'_1}\left(\int_Z |\nabla_z \psi(z)|^{p'_1}dz\right)^{1/p'_1}\\
&\leq C_{p_1}r_3^{d/p'_1}||\psi||_{W^{2,\frac{d+2}{2}}(Z)}\\
&\leq C_{p_1}r_3^{d/p'_1}||V||_{L^{\frac{d+2}{2}}(Z)},
\end{aligned}\end{equation}
where we have used the Sobolev embedding (since $p_1'\in (d+1,\infty)$ close to $d+1$, and then $p_1'< \frac{d\frac{d+2}{2}}{d-\frac{d+2}{2}}$), as well as $\varepsilon\leq r_3$ in the above inequality.
Similarly, in view of $(4.6)$ and $(3.11)$, there holds
\begin{equation}\begin{aligned}
&|I_5|\\
\leq& C\varepsilon\int_{t_0-r_3^2}^{t_0}\int_{\mathbb{R}^d}|\nabla_z
\psi^\varepsilon|\left(|\Gamma_\varepsilon(x_0,t_0;y,s)|+|c_i\Gamma_\varepsilon(x_i,t_0;y,s)|\right)|\nabla_y v_\varepsilon| \eta dyds\\
\leq&C\varepsilon(\frac{2r_2}{r_1})^{m-1} ||\nabla_y v_\varepsilon||_{L^{p_3}(\Omega')}||\nabla_z
\psi^\varepsilon||_{L^{p_4}({\Omega'})}\left(\int_{\tilde{\Omega}}\left({(t-s)^{-\frac{d}{2}p_2}} \exp \left\{-\frac{C|y|^{2}}{t-s}\right\}\right) dyds\right)^{1/p_2}\\
\leq& C\varepsilon(\frac{2r_2}{r_1})^{m-1} ||\nabla v_\varepsilon||_{L^{p_3}(\Omega'_{r_3,t_0})}||\nabla_z
\psi^\varepsilon||_{L^{p_4}(\tilde{\Omega}_{{r_3},t_0})}\left(\int_{t_0-r_3^2}^{t_0}{(t-s)^{-\frac{d}{2}p_2+\frac{d}{2}}} ds\right)^{1/p_2}\\
\leq& C \varepsilon r_3 (\frac{2r_2}{r_1})^{m-1}\sup_{\tilde{\Omega}_{r_3,t_0}}|v_\varepsilon|,
\end{aligned}\end{equation}
with $\Omega'=:E_{4r_3/5}\times (t_0-3r_3^2/4,t_0)$ and $\tilde{\Omega}=\tilde{\Omega}_{r_3,t_0}=:E_{r_3}\times (t_0-r_3^2,t_0)$, $p_2\in (1,+\frac{2}{d})$ close to $1+\frac{2}{d}$ and $p_3$ sufficiently large, close to $\infty$ with $\frac{1}{p_2}+\frac{1}{p_3}+\frac{1}{p_4}=1$. Note that $p_4>\frac{d+2}{2}$ close to $\frac{d+2}{2}$. And we have used, in the estimate $(4.9)$, the size estimate $|\Gamma_\varepsilon(x,t;y,s)|\leq C(t-s)^{-\frac{d}{2}}\exp\left\{-\frac{\kappa|x-y|^2}{t-s}\right\}$, $(4.8)$ as well as the the following estimates for $\nabla v_\varepsilon$,
\begin{equation}\begin{aligned}
\left(\fint_{L^{p_3}(\Omega'_{r_3,t_0})}|\nabla v_\varepsilon|^{p_3}\right)^{1/p_3} &\leq C r_3^{-1}\sup_{\tilde{\Omega}_{r_3,t_0}}|v_\varepsilon|+Cr_3\left(\fint_{\tilde{\Omega}_{r_3,t_0}}|(\mathcal{M}(V)-V(x/\varepsilon))
v_\varepsilon|^{\frac{d+2}{2}}\right)^{\frac{2}{d+2}}\\
& \leq C r_3^{-1}\sup_{\tilde{\Omega}_{r_3,t_0}}|v_\varepsilon|,
\end{aligned}\end{equation}
which may be proved by the estimates $|\nabla_y \Gamma_\varepsilon(x,t;y,s)|\leq C(t-s)^{-\frac{d+1}{2}}\exp\left\{-\frac{\kappa|x-y|^2}{t-s}\right\}$
and
$$\fint_{\tilde{\Omega}_{r_3,t_0}}|(\mathcal{M}(V)-V(x/\varepsilon))
|^{\frac{d+2}{2}}\leq C+\fint_{\tilde{\Omega}_{r_3,t_0}}|V(x/\varepsilon)
|^{\frac{d+2}{2}}\leq C+\fint_Z|V(z)
|^{\frac{d+2}{2}},$$
if $r_3\geq \varepsilon$.
Thus, combining $(4.7)$ and $(4.9)$ yields that
\begin{equation}
|I_3|\leq C \varepsilon r_3 (\frac{2r_2}{r_1})^{m-1}\sup_{\tilde{\Omega}_{r_3,t_0}}|v_\varepsilon|.
\end{equation}
Consequently, in view of $(3.37)$, we actually have
\begin{equation}\begin{aligned}
\sup_{E_{r_2}}|v_\varepsilon(\cdot,t_0)|\leq& C\left\{ \frac{(2r_2)^{m-1}}{{r_1}^{m-1}}\sup_{E_{r_1}}|v_\varepsilon(\cdot,t_0)|+\frac{(4r_2)^m}{r_3^{m}}\sup_{\tilde{\Omega}_{r_3,t_0}}|v_\varepsilon|\right.\\
&\quad\quad\left.+\frac{(2r_2)^m}{r_1^m}\varepsilon r_3^{-1}\log (2+\varepsilon^{-1}r_3)\sup_{\tilde{\Omega}_{r_3,t_0}}|v_\varepsilon|\right\}.
\end{aligned}\end{equation}
Then, totally similar to the discussion of Theorem 1.1, there holds
\begin{equation}
\sup_{E_{r_2}}|v_\varepsilon(\cdot,t_0)|\leq C\left\{ (\sup_{E_{r_1}}|v_\varepsilon(\cdot,t_0)|)^{\alpha}(\sup_{\tilde{\Omega}_{{r_3},t_0}}|v_\varepsilon|)^{1-\alpha}+\frac{r_3}{r_{1}}\left[\frac{\varepsilon}{{r_3}}\log (2+\frac{r_3}{{\varepsilon}})\right]^{\alpha}\sup_{\tilde{\Omega}_{{r_3},t_0}}|v_\varepsilon|\right\},
\end{equation}this together with $v_\varepsilon=u_\varepsilon e^{\mathcal{M}(V)t}$ and $-T<t<T$ gives the desired estimate $(1.11)$, thus completes this proof.
\begin{center}{\textbf{Acknowledgements}}
\end{center}

The author thanks Prof. Luis Escauriaza for helpful discussions.
\normalem\bibliographystyle{plain}{}
\bibliography{propagation}
\end{document}